\pdfmapfile{=<name>.map}
\documentclass[12pt,oneside,reqno,fleqn]{amsart}
\usepackage{amssymb}
\usepackage{bbm}
\usepackage{cases}
\usepackage{amsmath}
\usepackage{graphicx}
\usepackage{mathrsfs}
\usepackage{stmaryrd}
\usepackage{color}
\usepackage{soul}
\usepackage[dvipsnames]{xcolor}
\usepackage{amsfonts}
\usepackage{amsmath,amssymb,amsthm}
\usepackage{libertine}
\usepackage[T1]{fontenc}
\usepackage[libertine,varbb]{newtxmath}
\usepackage[colorlinks,linkcolor=red,citecolor=blue,urlcolor=red]{hyperref}
\usepackage[shortlabels] {enumitem}
\usepackage[nameinlink,capitalize]{cleveref}
\usepackage[spacing=true,kerning=true,tracking=true,letterspace=50]{microtype}
\numberwithin{equation}{section}
\newcommand{\be}{\begin{eqnarray}}
\newcommand{\ee}{\end{eqnarray}}
\newcommand{\ce}{\begin{eqnarray*}}
\newcommand{\de}{\end{eqnarray*}}
\newtheorem{theorem}{Theorem}[section]
\newtheorem{lemma}[theorem]{Lemma}
\newtheorem{proposition}[theorem]{Proposition}
\newtheorem{corollary}[theorem]{Corollary}
\theoremstyle{remark}
\newtheorem{assumption}[theorem]{Assumption}

\newtheorem{example}[theorem]{Example}
\newtheorem{remark}[theorem]{Remark}
\newtheorem{definition}[theorem]{Definition}

\crefname{eqn}{Equation}{Equations}
\crefname{assumption}{Assumption}{Assumptions}
\crefname{innercustomthm}{Condition}{Conditions}
\crefrangelabelformat{innercustomthm}{#3#1#4-#5#2#6}

\def\eps{\varepsilon}

\def\<{{\langle}}
\def\>{{\rangle}}
\def\({{\Big(}}
\def\){{\Big)}}

\def\bx{{\mathbf{x}}}

\def\={&\!\!=\!\!&}
\def\bt{\begin{theorem}}
\def\et{\end{theorem}}
\def\bl{\begin{lemma}}
\def\el{\end{lemma}}
\def\br{\begin{remark}}
\def\er{\end{remark}}
\def\bd{\begin{definition}}
\def\ed{\end{definition}}
\def\bp{\begin{proposition}}
\def\ep{\end{proposition}}
\def\bc{\begin{corollary}}
\def\ec{\end{corollary}}
\def\bx{\begin{example}}
\def\ex{\end{example}}

\def\cF{{\mathcal F}}

\def\cG{{\mathcal G}}

\def\cL{{\mathcal L}}

\def\cP{{\mathcal P}}

\def\mE{{\mathbb E}}
\def\E{\mE}

\def\mN{{\mathbb N}}

\def\PP{{\mathbb P}}

\def\geq{\geqslant}
\def\leq{\leqslant}

\def\div{\mathord{{\rm div}}}

\newcommand{\dd}{\,\mathrm{d}}
\newcommand{\loc}{\mathrm{loc}}
\newcommand{\R}{{\mathbb R}}

\newcommand{\norm}[1]{{\left\vert\kern-0.25ex\left\vert\kern-0.25ex\left\vert #1
    \right\vert\kern-0.25ex\right\vert\kern-0.25ex\right\vert}}

\allowdisplaybreaks

\begin{document}
	\title{A note on weak existence for singular SDEs}
	\date{\today}
	\author{Lucio Galeati}
\address{Lucio Galeati, 
EPFL, B\^atiment MA, 1015 Lausanne, Switzerland\newline
\indent Email:  lucio.galeati@epfl.ch
}

	\begin{abstract}
	Recently Krylov \cite{krylov1} established weak existence of solutions to SDEs for integrable drifts in mixed Lebesgue spaces, whose exponents satisfy the condition $1/q+d/p\leq 1$, thus going below the celebrated Ladyzhenskaya-Prodi-Serrin condition. We present here a variant of such result, whose proof relies on an alternative technique, based on a \textit{partial} Zvonkin transform; this allows for drifts with growth at infinity and/or in uniformly local Lebesgue spaces.\\[1ex]
	\\[1ex]
	\end{abstract}

\maketitle	

\section{Introduction}
Consider a multidimensional SDE on $\R^d$, $d\geq 2$, of the form
\begin{equation}\label{eq:intro-sde}
	\dd X_t = b_t(X_t) \dd t + \dd W_t.
\end{equation}
where $W$ is a standard Brownian motion.
It is by now well established that, even when the drift $b$ is singular, the SDE \eqref{eq:intro-sde} may still admits strong, pathwise unique solutions, in a regularization by noise fashion. In particular, a major focus in the literature is devoted to integrable drifts satisfying the Ladyzhenskaya-Prodi-Serrin condition, namely\footnote{See the end of the introduction for the definition of $L^q_t L^p_x$ and all other relevant function spaces.}
\begin{equation}\label{eq:LPS}
	b\in L^q_t L^p_x, \quad \frac{2}{q}+\frac{d}{p}\leq 1. \tag{LPS}
\end{equation}
The importance of \eqref{eq:LPS} comes from its connection to advection-diffusion equations, in particular the solvability of $3$D Navier--Stokes equations, as well as the fact that it arises naturally from a scaling argument (see e.g. \cite{BFGM2019}), hence why it is regarded as a \textit{critical class} of drifts for the solvability of \eqref{eq:intro-sde}.
The celebrated work of Krylov and R\"ockner \cite{krylov2005} came close to \eqref{eq:LPS}, up to only allowing the strict inequality and some additional technical constraints, which were later removed by X. Zhang in \cite{Zhang2011}; but it took several additional years and efforts to understand the critical case, see \cite{BFGM2019,krylov2021strong,RocZha2} and the review \cite{kinzebulatov2023}.

However recently Krylov \cite{krylov1} pointed out, elaborating on a previous result of Gy\"ongy and Mart\'inez \cite{GyoMar2001}, that in order to attain weak existence of solutions to \eqref{eq:intro-sde} it suffices to consider mixed Lebesgue spaces\footnote{More precisely, it is required that $b\in L^q_t L^p_x$ if $p\geq q$ and $b\in L^p_x L^q_t$ otherwise.} with exponents $p,q\in [1,\infty]$ satisfying
\begin{equation}\label{eq:krylov-exponents}
	\frac{1}{q}+\frac{d}{p} \leq 1.
\end{equation}
He also showed that this condition is optimal, in the sense that for $(p,q)$ satisfying the opposite inequality one can find drifts for which weak existence fails. Finer properties of the Markov process $X$ constructed in this way have then been established in \cite{krylov2,krylov3,krylov4,krylov5}.

This note stems from an attempt to understand condition \eqref{eq:krylov-exponents} from a different perspective, introducing an heuristic which hopefully might be relevant in other settings.
In order to explain it, it is useful to momentarily enlarge the class of problems and consider \eqref{eq:intro-sde} driven by a fractional Brownian motion $W$ of Hurst parameter $H\in (0,1)$. In this case, running the same scaling argument as in \cite{BFGM2019}, it was predicted in \cite[Section 1.1]{GalGer2022} that drifts $b\in L^q_t C^\alpha_x$ should be critical under the condition
\begin{equation}\label{eq:fbm-coefficients}
	\alpha = 1-\frac{1}{H q'}, \quad \frac{1}{q'}=1-\frac{1}{q},
\end{equation}
although a complete rigorous proof of this claim is still missing.
The scaling procedure consists in ``zooming in'' to look at the dynamics at short times; by self-similarity of the driving noise, this is equivalent (in law) to considering the same dynamics on $[0,1]$ but with rescaled drift $b^\lambda(t,x)=\lambda^{1-H} b(\lambda t, \lambda^H x)$. The critical class of drifts is then identified as the one invariant under this transformation, in the sense that $b$ and $b^\lambda$ have (roughly) the same norm; heuristically, the noise and the nonlinearity have ``the same strength'' and none is overtaking the other at small times.
In this sense, the scaling itself doesn't directly predict any wellposedness or illposedness results, rather it informs us on which component is locally driving the dynamics; if this is the drift $b$ (namely we are in the supercritical regime $\alpha<1-1/(Hq')$), then we might expect the dynamics to display similar phenomena as in the absence of noise. This a priori doesn't exclude it from being well-defined, or existence of solutions to hold, which still depends on the drift $b$ in consideration; but it tells us that the noise $W$ shouldn't be too much of help.

A different way to look at \eqref{eq:fbm-coefficients} is to regard it as an \textit{interpolation class} between two extrema, given respectively by $b\in L^1_t C^1_x$ ($q=1$) and $b\in L^\infty_t C^{1-1/H}_x$ ($q=\infty$)\footnote{Besov-H\"older spaces $C^\alpha_x$ are just one option and one might instead consider Lebesgue spaces with the same scaling behaviour. For instance, for $H=1/2$, $C^{-1}_x$ scales like $L^d_x$, which recovers the critical scale $b\in L^\infty_t L^d_x$. In this direction, let us mention \cite{butkovsky2023stochastic} for weak existence results for SDEs driven by fractional Brownian motion with (autonomous) drift in subcritical Lebesgue scales $L^p_x$.}.
Note that the endpoint $L^1_t C^1_x$ is the standard Cauchy-Lipschitz class, for which wellposedness of \eqref{eq:intro-sde} holds regardless of the choice of the driving noise $W$; instead the second endpoint $L^\infty_t C^{1-1/H}_x$, with a uniform-in-time regularity condition, is the one dictated by the scaling of the noise.

In this sense, if one is just interested in \textit{weak existence} of solutions, rather than their wellposedness, it makes sense to modify the first endpoint with another classical ODE requirement, $b\in L^1_t C^0_x$, under which solutions can be constructed by Peano's theorem (again, this result being valid for any choice of $W$). Interpolating between these two endpoints, one obtains a new class of drifts, for which there is some hope to retain weak existence results.
Observe that the range of exponents \eqref{eq:krylov-exponents} can be recovered by the same heuristics, interpolating between $L^1_t L^\infty_x$ for $q=1$ (``almost Peano'') and the time-homogeneous LPS class $L^\infty_t L^d_x$.
An analogue of \eqref{eq:krylov-exponents} in the fractional Brownian case is currently being obtained in \cite{ButGal}.

The aim of this note is to show that, in the case of Brownian SDEs, this interpolation heuristic can be made rigorous, by employing a \textit{partial Zvonkin transform}.
More precisely, given a drift which decomposes as $b=b^1+b^2$, where $b^1$ is a ``good drift'' for weak existence results, while $b^2$ is a more singular component, we can find a transformation $\Phi$ of the state space (obtained by solving a parabolic PDE) which removes the latter.
One then ends up with a new SDE for $Y=\Phi(X)$, driven by a drift $\tilde{b}$ which retains the properties of $b^1$ (e.g. local boundedness and linear growth); this allows to develop a priori estimates, which ultimately lead to existence by a compactness argument.

Although Zvonkin transform is by now a well-established tool for solving singular SDEs (see e.g. \cite{XXZZ2020}), it is usually performed at the level of the whole drift $b$, without isolating its most singular part. In this direction, the only precursors in the literature we are aware of are \cite{XieZhang} (where $b^1$ instead plays the role of a \textit{coercive} component) and partially \cite{ZhangYuan2021}.

For the sake of simplicity, so far we considered SDEs with additive noise, but our result allows for the presence of a multiplicative diffusion $\sigma$, satisfying the conditions outlined below. In the next statement, $\tilde{L}^p_x$ denote uniformly local Lebesgue spaces, see the notation section.

\begin{assumption}\label{ass:diffusion}
The drift $b:[0,T]\times \R^d\to \R^d$ is of the form $b=b^1+b^2$, where
\begin{equation}\label{eq:ass-drift}
	\frac{b^1}{1+|x|}\in L^{1+\eps}_t L^\infty_x, \quad b^2\in L^\infty_t \tilde{L}^{d+\eps}_x \quad \text{for some } \eps\in (0,1).
\end{equation}
The diffusion matrix $\sigma:[0,T]\times \R^d\to \R^{d\times d}$ is uniformly continuous in space, bounded and nondegenerate. Namely, there exist a constant $K>0$ such that
\begin{equation}\label{eq:ass-diffusion-1}
	K^{-1} |\xi|^2 \leq |\sigma^\ast (t,x) \xi|^2 \leq K|\xi|^2 \quad
	\forall\,\xi\in\R^d,\,(t,x)\in [0,T] \times \R^d.
\end{equation}
and a modulus of continuity $\omega_\sigma$ such that
\begin{equation}\label{eq:ass-diffusion-2}
	|\sigma(t,x)-\sigma(t,y)|\leq \omega_\sigma(|x-y|) \quad \forall (t,x,y)\in [0,T]\times \R^{2d}.
\end{equation}
\end{assumption}

To state our main result, we adopt the following solution concept for SDEs; $\cP(\R^d)$ denotes the set of probability measures on $\R^d$.

\begin{definition}\label{defn:weak-solution}
Let $b:[0,T]\times \R^d\to\R^d$ and $\sigma:[0,T]\times \R^d\to \R^{d\times d}$ be measurable functions, $\mu_0\in \cP(\R^d)$.
A \emph{weak solution} to the SDE
\begin{equation}\label{eq:intro-sde-detailed}
	\dd X_t = b_t(X_t) \dd t + \sigma_t(X_t) \dd W_t
\end{equation}
with initial law $\mu_0$ is a tuple $(\Omega,\cF, \{\cF_t\}_{t\geq 0}, \PP; X,W)$ given by a filtered probability space, a $\cF_t$-adapted process $X$ and a $\cF_t$-Brownian motion $W$ such that
\begin{align*}
	\int_0^T |b_s(X_s)| \dd s + \int_0^T |\sigma_s(X_s)|^2 \dd s <\infty \quad \PP\text{-a.s.},
\end{align*}
$X_0$ is distributed as $\mu_0$ and $\PP$-a.s. it holds
\begin{align*}
	X_t =X_0 + \int_0^t b_s(X_s) \dd s + \int_0^t \sigma_s(X_s) \dd W_s \quad \forall\, t\in [0,T].
\end{align*} 
\end{definition}

Let us set $\cP_1(\R^d):=\{\nu\in \cP(\R^d): \int_{\R^d} |x| \nu(\!\dd x) <\infty \}$.

\begin{theorem}\label{thm:main-theorem}
Let $(b,\sigma)$ satisfy Assumption \ref{ass:diffusion}; then for any initial distribution $\mu_0\in \mathcal{P}_1(\R^d)$, there exists a weak solution $X$ to the SDE \eqref{eq:intro-sde-detailed}, with initial law $\mu_0$, in the sense of Definition \ref{defn:weak-solution}.
\end{theorem}

Here are two relevant consequences of Theorem \ref{thm:main-theorem}.

\begin{corollary}\label{cor:main-cor-1}
Let $b\in L^q_t \tilde{L}^p_x$ for parameters $(p,q)\in [1,\infty]$ satisfying
\begin{equation}\label{eq:exponents-weaker}
	\frac{1}{q}+\frac{d}{p}<1
\end{equation}
Then $b$ admits a decomposition satisfying \eqref{eq:ass-drift}, so that Theorem \ref{thm:main-theorem} applies.
\end{corollary}

To state the next corollary, we need to define weak solutions to Fokker-Planck equations.

\begin{definition}\label{defn:FP}
Given measurable $b$, $\sigma$, set $a:=\sigma\sigma^\ast$. We say that a flow of measures $t\mapsto \mu_t$ is a weak solution to the Fokker-Planck equation
\begin{equation}\label{eq:FP}
	\partial_t \mu + \div (b \mu) = \frac{1}{2} \sum_{i,j} \partial^2_{ij} (a_{ij} \mu)
\end{equation}
if $t\mapsto \mu_t$ is continuous in the sense of distributions, $b\mu$ and $a_{ij}\mu$ are well defined distributions and for any $\varphi\in C^{\infty}_c((0,T)\times \R^d)$ it holds
\begin{equation}\label{eq:FP-weak}
	\int_{[0,T]} \int_{\R^d} (\partial_t \varphi + b_t\cdot\nabla \varphi + \frac{1}{2}\sum_{i,j} a_{ij} \partial^2_{ij} \varphi)(x) \mu_t(\!\dd x) \dd t = 0.
\end{equation}
\end{definition}

\begin{corollary}\label{cor:main-cor-2}
Let $b$, $\sigma$ satisfy Assumption \ref{ass:diffusion}. Then for any $\mu_0\in \cP_1$, there exists a weak solution $\mu$ to the Fokker-Planck equation \eqref{eq:FP} in the sense of Definition \ref{defn:FP}, with the properties that $t\mapsto \mu_t$ is continuous in the weak topology of measures and $\mu\vert_{t=0}=\mu_0$. Moreover $\mu\in L^{\tilde q}_t L^{\tilde p}_{x}$ for any $(\tilde p, \tilde q)$ satisfying
\begin{equation}\label{eq:integrability-density}
\frac{1}{\tilde q} + \frac{d}{\tilde p} >d, \quad (\tilde p, \tilde q)\in (1,\infty)^2.
\end{equation}
In particular, this ensures that $b \mu,\, a \mu \in L^1_t L^1_{loc}$.
\end{corollary}

Let us give some comments on Theorem \ref{thm:main-theorem}.

\begin{remark}\label{rem:comparison-krylov}
Our result presents both advantages and drawbacks compared to the original one from \cite{krylov1}.
On one hand, we can only allow a strict inequality in \eqref{eq:exponents-weaker}, as a consequence of the parameter $\eps>0$ in \eqref{eq:ass-drift}; on the other, we can allow for drifts being either unbounded (at most of linear growth) or belonging to localised Lebesgue spaces $\tilde L^p_x$.
Finally, contrary to \cite{krylov1}, our condition $b\in L^q_t \tilde L^p_x$ doesn't change depending on whether $q\leq p$ or $p\leq q$, which makes it slightly more natural in analogy with \eqref{eq:LPS}.
\end{remark}

\begin{remark}\label{rem:nonuniqueness}
Both the result from \cite{krylov1} and Theorem \ref{thm:main-theorem} only establish weak existence of solutions.
In fact, counterexamples to uniqueness in law in Besov spaces have been constructed in \cite[Section 1.3]{GalGer2022}; by Remark 1.8 therein, for any choice of $(p,q)\in [1,\infty]^2$ satisfying
\begin{align*}
	\frac{2}{q}+\frac{d}{p}>1,\quad p>d,
\end{align*}
one can construct a drift $b\in L^q_t L^p_x$ for which there is non-uniqueness in law for \eqref{eq:intro-sde}.
\end{remark}

\begin{remark}\label{rem:navier-stokes}
In light of Remark \ref{rem:nonuniqueness}, it might seem that condition like \eqref{eq:krylov-exponents} is not so interesting; however it might have relevant applications for nonlinear PDEs.
To illustrate this, let us consider the prototypical case of the $3$D Navier--Stokes equations (the same which motivated the interest in \eqref{eq:LPS}); we refer to \cite{RoRoSa2016} for a general overview and to \cite{ZhaZha2021} for a discussion of the associated Lagrangian SDE.
Leray weak solutions satisfy $u\in L^\infty_t L^2_x\cap L^2_t H^1_x$, which by Sobolev embeddings implies
\begin{equation}\label{eq:regularity-navier-stokes}
	u\in L^q_t L^p_x\quad \text{for} \quad \frac{2}{q}+\frac{3}{p} = \frac{3}{2} \quad
	\forall\, p\in [2,6]
\end{equation}
which is considerably far from \eqref{eq:LPS}.
However, by taking $q=2$, $p=6$, condition \eqref{eq:regularity-navier-stokes} interesects with \eqref{eq:krylov-exponents}, allowing to invoke the results from \cite{krylov1} to obtain a priori estimates for the associated SDE.
It should be also mentioned that, exploiting the fact that $u$ is divergence free and its Sobolev regularity, recently \cite{ZhaZha2021} and \cite{zhao2020} were able to construct weak solutions and prove uniqueness of the stochastic Lagrangian flow associated to $u$. In this sense, condition \eqref{eq:krylov-exponents} is just another small piece of the puzzle, hinting that \eqref{eq:LPS} might not be the end of the story for Navier-Stokes equations.
\end{remark}

\begin{remark}\label{rem:generalizations}
We expect our strategy to work in other cases, for instance: i) $b$ of the form $b=b^1+\ldots+b^n$ with $b^i\in L^{q_i}_t \tilde L^{p_i}_x$ with $(q_i,p_i)$ satisfying \eqref{eq:exponents-weaker}; ii) coefficients belonging to mixed normed spaces, i.e. $b\in L^q_t L^{p_1}_{x_1} \ldots L^{p_d}_{x_d}$ with $1/q+\sum_i 1/p_i<1$, in analogy to what was obtained in \cite{ling2021strong} as a refinement of \cite{krylov2005}.

Something more interesting would be to understand whether one can obtain novel existence and/or uniqueness results by interpolating other classes of drifts. For instance, one could consider $L^1_t L^\infty_x$ and $L^\infty_t C^\gamma_x$ with $\gamma>-1/2$, where for the latter weak existence and uniqueness of solutions was established in \cite{FlIsRu2017} again by Zvonkin transform. We leave this problem for future investigations.
\end{remark}

\subsection*{Structure of the paper}
We conclude this introduction by explaining the relevant notations and conventions.
In Section \ref{sec:preliminaries} we recall some analytic tools, most notably involving the resolution of parabolic PDEs, invoked throughout the paper. We develop all the relevant a priori estimates for our SDE in Section \ref{sec:apriori}, by first considering smooth coefficients; then in Section \ref{sec:main-thm} we prove our results, by running a compactness argument and passing to the limit.

\subsection*{Notations and conventions}
We always work on a finite time interval $[0,T]$, although arbitrarily large. We write $a\lesssim b$ to mean that there exists a positive constant $C$ such that $a \leq C b$; we use the index $a\lesssim_\lambda b$ to highlight the dependence $C=C(\lambda)$.

For any $m\in \mathbb{N}$ and $p\in [1,\infty]$, we denote by $L^p(\R^d;\R^m)$ the standard Lebesgue space; when there is no risk of confusion in the parameter $m$, we will simply write $L^p_x$ for short and denote by $\| \cdot\|_{L^p_x}$ the corresponding norm. Similarly, we denote by $L^p_{loc}(\R^d;\R^m)=L^p_{loc}$ local Lebesgue spaces, endowed with their natural Frechét topology; finally, we consider uniformly local Lebesgue spaces $\tilde L^p(\R^d;\R^m)=\tilde L^p_x$ as defined by 
\begin{align*}
	\tilde{L}^p_x:=\Big\{\varphi \in L^p_{loc}(\R^d;\R^m): \|
	\varphi\|_{\tilde L^p_x}:= \sup_{z\in\R^d} \| \chi^z \varphi \|_{L^p_x}<\infty\Big\};
\end{align*}
here $\chi^z:=\chi(\cdot - z)$, where $\chi$ is a smooth nonnegative function such that $\chi(x)= 1$ if $|x|\leq 1$ and $\chi(x)=0$ if $|x|\geq 2$.
By a finitely covering technique, one can check that the definition of $\tilde L^p_x$ does not depend on $\chi$, cf. \cite[p. 5193]{XXZZ2020}.
Similarly, one can check by covering arguments that for $p=\infty$, the spaces $\tilde L^\infty_x$ and $L^\infty_x$ coincide with equivalent norms:
\begin{equation}\label{eq:uniform_local_Linfty}
	\| \varphi\|_{\tilde L^\infty_x}= \sup_{z\in\R^d} \| \chi^z \varphi\|_{L^\infty_x} \sim_\chi \|\varphi\|_{L^\infty_x}.
\end{equation}
$H^1_x=H^1(\R^d;\R^m)$ is used to denote the Sobolev space of functions in $L^2_x$ whose weak differential is also in $L^2_x$.

For $\alpha\in [0,+\infty)$, $C^\alpha (\mathbb{R}^d;\R^m)=C_x^\alpha$ stands for the usual H\"older continuous function space, made of continuous bounded functions with continuous and bounded derivatives up to order $\lfloor \alpha\rfloor\in\mathbb{N}$ and with globally $\{\alpha\}$-H\"older continuous derivatives of order $\lfloor \alpha\rfloor$. Similarly to the case $L^\infty_x$ treated above in \eqref{eq:uniform_local_Linfty}, if we defined a uniformly local H\"older space $\tilde C^\alpha_x$, we would still end up with $C^\alpha_x$:
\begin{equation}\label{eq:uniform_local_holder}
	\| \varphi\|_{\tilde C^\alpha_x} := \sup_{z\in\R^d} \| \chi^z \, \varphi\|_{C^\alpha_x}\sim_{\alpha,\chi} \| \varphi\|_{C^\alpha_x}.
\end{equation}
In other words, for $L^\infty_x$-based spaces, uniformly local and global estimates coincide.

Given a Banach space $E$, we denote by $C([0,T];E)=C_t E$ the set of all continuous functions $\varphi:[0,T]\to E$, endowed with the supremum norm $\| \varphi\|_{C^0_t E}=\sup_{t\in [0,T]} \| \varphi_t\|_E$. Similarly for $\gamma\in (0,1)$ we define $C^\gamma([0,T];E)=C^\gamma_t E$ as the set of $\gamma$-H\"older continuous functions, with associated seminorm and norm
\begin{equation*}
	\llbracket \varphi \rrbracket_{C^\gamma_t E}:=\sup_{s\neq t} \frac{\| \varphi_t-\varphi_s\|_E}{|t-s|^\alpha}, \quad
	\| \varphi \|_{C^\gamma_t E} := \| \varphi \|_{C^0_t E} +  \llbracket \varphi \rrbracket_{C^\gamma_t E}.
\end{equation*}

Given a Frechét space $E$, with topology induced by a countable collection of seminorms $(d_j)_{j\in\mN}$, and a parameter $q\in [1,\infty]$, we denote by $L^q(0,T;E)=L^q_t E$ the space of measurable functions $\varphi:[0,T]\to E$ such that $\int_0^T d_j( \varphi_t, 0)^q \dd t <\infty$ for all $j\in\mN$ (with the usual convention for $q=\infty$).
Similarly, we say that $\varphi^n\to \varphi$ in $L^q_t E$ if 
\begin{align*}
	\lim_{n\to\infty} \int_0^T d_j( \varphi^n_t, \varphi_t)^q \dd t <\infty \quad \forall\, j\in \mN.
\end{align*}
The above definitions can be concatenated by choosing different $E$, so that one can define $C^\gamma_t C^0_x$, $L^\infty_t C^1_x$, $L^q_t \tilde L^p_x$ and so on. Whenever $q=p$, we might write for simplicity $L^p_{t,x}$ in place of $L^p_t L^p_x$.
When $E=\R^d$, for simplicity we will drop it and just write $L^q_t$, $C^\gamma_t$, in place of $L^q_t \R^d$, $C^\gamma_t \R^d$.

Whenever we are given a filtered probability space $(\Omega,\mathcal{F},\{\cF_t\}_{t\geq 0},\mathbb{P})$, we will always assume the filtration $\{\cF_t\}_{t\geq 0}$ to satisfy the standard assumptions. We denote by $\E$ expectation w.r.t. $\PP$; if $X$ is a random variable define on $\Omega$, we denote by $\cL(X)=\PP\circ X^{-1}$ its law under $\PP$.

\section{Analytic preliminaries}\label{sec:preliminaries}

As mentioned above, a primary tool in our analysis is the so called Zvonkin transformation, which is related to solving a class of backward parabolic PDEs of the form
\begin{equation}\label{eq:zvonkin-pde}
\partial_t u + \frac{1}{2}a:D^2 u + g\cdot\nabla u -\lambda u= -f, \quad u\vert_{t=T}=0.
\end{equation}
Here we assume we are given $\sigma$ satisfying conditions \eqref{eq:ass-diffusion-1}-\eqref{eq:ass-diffusion-2} and we define the associated parameter set $\Theta:=(T,d,K,\omega_\sigma)$; we adopt the notations $a=\sigma \sigma^\ast$, $a:D^2 u = \sum_{i,j} a_{ij} \partial^2_{ij} u$ and $g\cdot \nabla u=\sum_{i}g_i\partial_i u$. If $u$ and $f$ are vector-valued, then \eqref{eq:zvonkin-pde} is understood componentwise.

\begin{proposition}\label{prop:zvonkin}
Let $\sigma$ satisfy \eqref{eq:ass-diffusion-1}-\eqref{eq:ass-diffusion-2}, $\eps>0$ and $g\in L^\infty_t\tilde L^{d+\eps}_x$.
Then there exists $\lambda_0\geq 1$, depending on $\Theta$, $\eps$ and $\| g \|_{L^{\infty}_t \tilde L^{d+\eps}_x}$, such that for all $\lambda\geq \lambda_0$ and for all $f\in L^\infty_t \tilde L^{d+\eps}_x$ there exists a unique strong solution $u$ to the PDE \eqref{eq:zvonkin-pde}.
Furthermore there exist $\delta=\delta(\eps)>0$ and $C=C(\Theta,\eps,\| g \|_{L^{\infty}_t \tilde L^{d+\eps}_x})$ such that
\begin{equation}\label{eq:zvonkin-estim}
\lambda^\delta \| u\|_{C^0_t C^1_x} + \| u\|_{C^{1/2}_t C^0_x} \leq C \| f\|_{L^\infty_t \tilde L^{d+\eps}_x}.
\end{equation}
%
\end{proposition}

\begin{proof}
Although the result is classical in the case of constant diffusion and classical Lebesgue spaces, we haven't found a direct reference in our setting; we will derive it from \cite[Theorem 3.2]{XXZZ2020}, which however makes the proof a bit technical.
Following \cite{XXZZ2020}, we will employ the spaces $\tilde{H}^{\alpha,p}$, $\tilde{\mathbb{H}}^{\alpha,p}_q(T)$ and $\tilde{\mathbb{L}}^p_q(T)$; we refer the reader to \cite[pp. 5192-3]{XXZZ2020} for their exact definition.

By time reversal, we can reduce ourselves to the case of a forward parabolic equation with $u\vert_{t=0}=0$.
By the hypothesis, we can find $q\in (1,\infty)$ large enough and $\alpha>1$ such that $2\delta:= 2-\alpha-2/q-d/(d+\eps)>0$; applying \cite[Theorem 3.2]{XXZZ2020} for such $\alpha$ and $q_1=q$, $p_1=p=d+\eps$, $p'=q'=\infty$, we deduce the wellposedness of \eqref{eq:zvonkin-pde} as well as the estimate
\begin{align*}
\lambda^\delta \| u\|_{\tilde{\mathbb{H}}^{\alpha,\infty}_\infty(T)} + \| \partial_t u\|_{\tilde{\mathbb{L}}^{d+\eps}_q(T)} + \| u\|_{\tilde{\mathbb{H}}^{2,d+\eps}_q(T)}
\lesssim \| f\|_{\tilde {\mathbb{L}}^{d+\eps}_q(T)}
\lesssim \| f\|_{L^\infty_t \tilde L^{d+\eps}_x}
\end{align*}
The estimate for $\| u\|_{C^0_t C^1_x}$ then follows from the embedding $\mathbb{H}^{\alpha,\infty}_\infty(T)\hookrightarrow L^\infty_t C^1_x$, since $\alpha>1$.

The estimate for $\| u\|_{C^{1/2}_t C^0_x}$ instead follows by interpolation arguments. Set $\theta=1/2+1/q$; then using \eqref{eq:uniform_local_holder} for $\alpha=0$, Sobolev embeddings and interpolation inequalities, for any $s\leq t$ it holds
\begin{align*}
	\| u_t-u_s\|_{C^0_x}
	& \lesssim \sup_{z\in \R^d} \| \chi^z (u_t-u_s)\|_{C^0_x}
	\leq \sup_{z\in \R^d} \| \chi^z (u_t-u_s)\|_{C^{1-2/q-d/(d+\eps)}_x}\\
	& \lesssim \| u_t-u_s\|_{\tilde H^{1-2/q,d+\eps}}
	\lesssim |t-s|^{1/2} \| \partial_t u\|_{\tilde{\mathbb{L}}^{d+\eps}_q(T)}^{\theta}\, \| u\|_{\tilde{\mathbb{H}}^{2,d+\eps}_q(T)}^{1-\theta}\\
	& \lesssim |t-s|^{1/2}\, \| f\|_{L^\infty_t \tilde{L}^{d+\eps}_x}.
\end{align*}
In particular, in the intermediate passage we used that (again by interpolation inequalities and Sobolev embeddings) for any $g$ it holds
\begin{align*}
	\llbracket g\rrbracket_{C^{1/2}_t H^{1-2/q,d+\eps}_x}
	\lesssim \| g\|_{\dot W^{1/2+1/q,q}_t L^{d+\eps}_x}
	\lesssim \| \partial_t g\|_{L^q_t L^{d+\eps}_x}^{\theta}\, \| g\|_{L^q_t H^{2,d+\eps}_x}^{1-\theta} \quad \text{for} \quad \theta=\frac{1}{2}+\frac{1}{q}.
\end{align*}
Combining the previous estimate with the fact that $u\vert_{t=0}=0$ readily yields the bound for $\| u\|_{C^{1/2}_t C^0_x}$.
\end{proof}
Let $b=b^1+b^2$ and $\sigma$ as in Assumption \ref{ass:diffusion}. By virtue of estimate \eqref{eq:zvonkin-estim}, we can find $\bar{\lambda}=\bar\lambda(\Theta,\eps,\| b^2\|_{L^\infty_t \tilde{L}^{d+\eps}_x})$ such that the vector-valued solution $u:=u^b$ to the PDE \eqref{eq:zvonkin-pde} associated to $f=g=b^2$ and $\bar\lambda$ satisfies $\| u\|_{C^0_t C^1_x} \leq 1/2$.
Correspondingly, we define the \emph{partial Zvonkin transform} associated to $b$ to be $\Phi_t(x):= x + u^b_t(x)$.

\begin{lemma}\label{lem:properties-zvonkin}
For any $t\in [0,T]$, $\Phi_t$ is a diffeomorphism of $\R^d$ into itself and there exists a constant $C=C(\Theta,\eps,\| b^2\|_{L^\infty_t \tilde{L}^{d+\eps}_x})$ such that for all $x,\,y\in \R^d$ and $s,\,t\in [0,T]$ it holds
\begin{equation}\label{eq:properties-zvonkin}
	\frac{1}{2} |x-y|\leq |\Phi_t(x)-\Phi_t(y)| \leq 2 |x-y|, \quad 
	|\Phi_t(x)-\Phi_s(x)|\leq C |t-s|^{1/2}.
\end{equation}
Moreover the same estimate holds with $\Phi_t$ replaced by its inverse $\Phi^{-1}_t$.
\end{lemma}

\begin{proof}
The statement for $\Phi_t$ follows by its definition and the available estimates for $u^b$: it holds $\nabla\Phi_t(x)=I + \nabla u^b_t(x)$ with $|\nabla u^b_t(x)|\leq 1/2$, yielding the diffeomorphism property and the first estimate in \eqref{eq:properties-zvonkin}, while  $|\Phi_t(x)-\Phi_s(x)|=|u^b_t(x)-u^b_s(x)|\leq |t-s|^{1/2} \| u^b\|_{C^{1/2}_t C^0_x}$.
The bi-Lipschitz property for $\Phi^{-1}_t$ follows similarly; we are left with estimating the H\"older continuity of $t\mapsto \Phi^{-1}_t(x)$. It holds
\begin{align*}
	\sup_x |\Phi^{-1}_t(x)-\Phi^{-1}_s(x)|
	& = \sup_x | x - \Phi^{-1}_s(\Phi_t(x))|
	= \sup_x | \Phi^{-1}_s(\Phi_s(x)) - \Phi^{-1}_s(\Phi_t(x))|\\
	& \leq 2 \sup_x | \Phi_s(x)-\Phi_t(x)| \lesssim |t-s|^{1/2} \| u^b\|_{C^{1/2}_t C^0_x}. \qquad \qedhere
\end{align*}
\end{proof}
We conclude this section with a basic result, guaranteeing that any $f$ belonging in mixed Lebesgue spaces can be decomposed as in \eqref{eq:ass-drift}.
\begin{lemma}\label{lem:decomposition-lemma}
Let $f\in L^q_t L^p_x$ for some $(q,p)\in [1,\infty]$ satisfying $1/q+d/p<1$. Then there exists $\eps=\eps(p,q)>0$ such that $f$ can be decomposed as $f=f^\leq+f^>$, where
\begin{align*}
	f^\leq \in L^{1+\eps}_t L^\infty_x,
	\quad \| f^{\leq} \|_{L^{1+\eps}_t L^\infty_x} \leq \| f\|_{L^q_t L^p_x}^{\frac{q}{1+\eps}},
	\quad f^>\in L^\infty_t L^{d+\eps}_x,
	\quad \| f^>\|_{L^\infty_t L^{d+\eps}_x} \leq 1.
\end{align*}
A similar statement holds with $ L^p_x$ (resp. $ L^{d+\eps}_x$) replaced by $\tilde L^p_x$ (resp. $\tilde L^{d+\eps}_x$).
\end{lemma}

\begin{proof}
For notational simplicity, we give the proof in the case $f\in L^q_t L^p_x$, the other case being identical up to keeping track of $\chi^z$ in all the computations. The result is a basic consequence of interpolation theory, but let us give an explicit choice of the decomposition.
By the assumption, we can find $\eps>0$ such that
\begin{equation}\label{eq:interpolation-coeff}
	\frac{1+\eps}{q}+ \frac{d+\eps}{p}=1.
\end{equation}
For such choice, set
\begin{align*}
	f^\leq_t(x):= f_t(x) \mathbbm{1}_{|f_t(x)|\leq R_t}, \quad
	f^>_t(x):= f_t(x) \mathbbm{1}_{|f_t(x)|> R_t}, \quad
	R_t := \| f_t\|_{L^p_x}^{\frac{p}{p-d-\eps}}.
\end{align*}
Then it holds
\begin{align*}
	\| f^>_t\|_{L^{d+\eps}_x}
	\leq \Big( \int_{\R^d} R_t^{d+\eps-p} |f_t(x)|^p \dd x \Big)^{\frac{1}{d+\eps}}
	\leq R_t^{\frac{d+\eps-p}{d+\eps}} \| f_t\|_{L^p_x}^{\frac{p}{d+\eps}}=1 \quad \forall\, t\in [0,T]
\end{align*}
while by virtue of \eqref{eq:interpolation-coeff} we have
\begin{align*}
	\int_0^T \| f^\leq_t\|_{L^\infty_x}^{1+\eps} \dd t
	\leq \int_0^T R_t^{1+\eps} \dd t = \int_0^T \| f_t\|_{L^p_x}^q \dd t < \infty.  \qquad \qedhere
\end{align*}
\end{proof}

\section{A priori estimates}\label{sec:apriori}

Throughout this section, we will assume that, in addition to Assumption \ref{ass:diffusion}, $b$ and $\sigma$ are sufficiently regular; to fix the ideas, we will take $\sigma$ uniformly Lipschitz and $b\in L^1_t C^1_{\loc}$ such that $b/(1+|x|)\in L^1_t L^\infty_x$.
In this case, strong existence and pathwise uniqueness of solutions to \eqref{eq:intro-sde-detailed} is classical\footnote{Many classical textbooks, see for instance \cite[Thms. 5.2.5-5.2.9]{KarShr}, only state the result in the case of uniform-in-time bounds; it is however easy to check that the same type of Gr\"onwall estimates allow for time-dependent, $L^1$-integrable weights. In a slightly different setup, see for instance \cite[Sec. 3.2.3]{GaHaMa2022}.}; our goal is to devise a priori estimates which only rely on the norms and parameters appearing in Assumption \ref{ass:diffusion}.
Recall that we are also supplied with a random initial condition $X_0$ satisfying $\E[|X_0|]<\infty$ (corresponding to $\mu_0\in \cP_1$).

We divide our analysis in Lemmas \ref{lem:moments-X} and \ref{lem:apriori-density} below.

\begin{lemma}\label{lem:moments-X}
There exists a constant $C$, depending on $\Theta$, $\eps$, $\| b^2\|_{L^\infty_t \tilde{L}^{d+\eps}_x}$, $\| b^1/(1+|x|)\|_{L^{1+\eps}_tL^\infty_x}$ and $T$, such that
\begin{equation}\label{eq:moments-X}
	\E\big[ \| X\|_{C^{\eps/(1+\eps)}_t} \big] \leq C\big(1 + \E\big[|X_0|\big]\big).
\end{equation}
\end{lemma}

\begin{proof}
\emph{Step 1: Partial Zvonkin transform}.
Let $u^b$ be defined as in Section \ref{sec:preliminaries} for suitably chosen $\bar\lambda$ and set $\Phi_t(x):=x + u_t^b(x)$. By Lemma \ref{lem:properties-zvonkin}, $\Phi_t$ is a diffeomorphism from $\R^d$ to itself; moreover since $u^b$ solves \eqref{eq:zvonkin-pde} for $f=g=b^2$, by construction $\Phi$ solves the PDE
\begin{align*}
	\partial_t \Phi + \frac{1}{2}a:D^2 \Phi + b^2\cdot\nabla \Phi = \bar\lambda u, \quad \Phi_T(x)=x.
\end{align*}
Introducing the new variable $Y_t=\Phi_t(X_t)$, we deduce that $Y$ solves
\begin{align*}
\dd Y_t
& = (\partial_t \Phi + \frac{1}{2}a:D^2\Phi + b\cdot\nabla \Phi)_t(X_t) \dd t + \nabla \Phi_t (X_t) \sigma_t(X_t) \dd W_t\\
& = (\bar\lambda u + b^1\cdot\nabla \Phi)_t(X_t) \dd t + \nabla \Phi_t(X_t)\sigma_t(X_t) \dd W_t
\end{align*}
so that $Y$ solves the SDE $\dd Y = \tilde b(Y)\dd t + \tilde\sigma(Y)\dd W$ with new coefficients
\begin{align*}
	\tilde b := (\bar\lambda u + b^1\cdot\nabla \Phi)\circ \Phi^{-1},
	\quad \tilde \sigma := (\nabla \Phi_t \sigma_t)\circ \Phi^{-1}.
\end{align*}
It follows from the smallness condition $\|u\|_{C^0_t C^1_x} \leq 1/2$ and property \eqref{eq:properties-zvonkin} (applied both for $\Phi_t$ and $\Phi_t^{-1}$) that $\tilde b$ is still of linear growth, and in particular
\begin{equation}\label{eq:estim-new-coefficients}
	\Big\| \frac{\tilde b_t}{1+|x|}\Big\|_{L^\infty_x} \leq \bar\lambda  + 4\, \Big\| \frac{b^1_t}{1+|x|}\Big\|_{L^\infty_x} \quad \forall\, t\in [0,T], \quad 
	\| \tilde{\sigma}\|_{L^\infty_{t,x}} \leq 2\, \| \sigma\|_{L^\infty_{t,x}}.
\end{equation}
Let us set $h_t:= \bar\lambda  + 4\, \| b^1_t/(1+|x|)\|_{L^\infty_x}$; by Assumption \ref{ass:diffusion}, it holds $h\in L^{1+\eps}_t$.

\emph{Step 2: A priori estimates for $Y$.}
Set $Z_t:= \int_0^t \tilde\sigma_s(Y_s) \dd W_s$, so that $Y$ satisfies
\begin{align*}
	Y_t = Y_0 + \int_0^t \tilde{b}_s(Y_s) \dd s + Z_t;
\end{align*}
since $\tilde{b}_s(x)\leq h_s(1+|x|)$, we can apply Gr\"onwall's lemma at a pathwise level to find
\begin{equation}\label{eq:pathwise-gronwall}
	\| Y(\omega)\|_{C^0_t} \leq e^{\| h \|_{L^1_t}} \Big( \| h \|_{L^1_t} + |Y_0(\omega)| + \sup_{t\in [0,T]} |Z_t(\omega)| \Big) \quad \PP\text{-a.s.}
\end{equation}
Furthermore by the properties of $\tilde{b}$ and H\"older's inequality, it holds
\begin{align*}
	|Y_t-Y_s|
	& \leq (1+\| Y\|_{C^0_t}) \int_s^t h_r \dd r + |Z_t-Z_s|\\
	& \lesssim |t-s|^{\frac{\eps}{1+\eps}} \Big( \| h\|_{L^{1+\eps}_t} + \| h\|_{L^{1+\eps}_t}\, \| Y\|_{C^0_t} + \llbracket Z\rrbracket_{C^{\eps/(1+\eps)}_t} \Big) \quad \PP\text{-a.s.};
\end{align*}
dividing by $|t-s|^{\eps/(1+\eps)}$, taking supremum and combining this with \eqref{eq:pathwise-gronwall}, one arrives at
\begin{equation}\label{eq:pathwise-holder-estim}
	\| Y(\omega) \|_{C^{\eps/(1+\eps)}_t} \lesssim e^{2 \| h \|_{L^{1+\eps}_t} } \Big( 1 + |Y_0(\omega)| + \|Z(\omega)\|_{C^{\eps/(1+\eps)}_t} \Big)\quad \PP\text{-a.s.}
\end{equation}

\emph{Step 3: A priori estimates for $X$.}
Recall that $X_t=\Phi_t^{-1}(Y_t)$, where by construction $\Phi^{-1}$ satisfies \eqref{eq:properties-zvonkin}; moreover it holds $|\Phi_t^{-1}(x)|\leq |x| + 1/2$ uniformly in $t$, since
\begin{align*}
	|x|=|\Phi_t(\Phi_t^{-1}(x))| = |\Phi_t^{-1}(x) + u_t(\Phi_t^{-1}(x))|
	\geq |\Phi_t^{-1}(x)| - \| u_t\|_{C^0_x} \geq |\Phi_t^{-1}(x)| - \frac{1}{2}.
\end{align*}
It follows that $\PP$-a.s. $\sup_t |X_t| \leq 1 + \sup_t |Y_t|$ and
\begin{align*}
	|X_t-X_s|
	\leq |\Phi_t^{-1} (Y_t)-\Phi_t^{-1}(Y_s)| + |\Phi_t^{-1} (Y_s)-\Phi_s^{-1}(Y_s)|
	\lesssim |Y_t-Y_s| + |t-s|^{1/2};
\end{align*}
combined with the pathwise bounds \eqref{eq:pathwise-gronwall}-\eqref{eq:pathwise-holder-estim}, we finally obtain an estimate of the form
\begin{equation}\label{eq:pathwise-estim-X}
	\| X(\omega)\|_{C^{\eps/(1+\eps)}_t} \lesssim 1 + |X_0(\omega)| + \|Z(\omega)\|_{C^{\eps/(1+\eps)}_t}\quad \PP\text{-a.s.}
\end{equation}
where the hidden constant depends on $\Theta$, $\eps$, $\| b^2\|_{L^\infty_t \tilde{L}^{d+\eps}_x}$, $T$ and $\| h\|_{L^{1+\eps}_t}$.

Recall that $Z$ is defined as a stochastic integral, with uniformly bounded $\tilde{\sigma}$; a standard application of Burkholder-Davis-Gundy inequality and Kolmogorov's continuity theorem allows to deduce that $\|Z\|_{C^{\eps/(1+\eps)}_t}$ admits moments of any order, in particular it has finite expectation.
In view of the assumptions on $X_0$,
this concludes the proof.
\end{proof}

As a next step, we derive a priori estimates on the density of $\cL(X_t)$.

\begin{lemma}\label{lem:apriori-density}
Let $b$, $\sigma$ be regular coefficients satisfying Assumption \ref{ass:diffusion}, $X$ the solution to \eqref{eq:intro-sde-detailed} and set $\mu_t=\cL(X_t)$. Then for any pair $(\tilde p, \tilde q)$ satisfying \eqref{eq:integrability-density} it holds
\begin{equation*}
	\| \mu\|_{L^{\tilde q}_t L^{\tilde p}_x} \lesssim 1+ \E[|X_0|]
\end{equation*}
where the hidden constant depends on $\Theta$, $\eps$, $\tilde p$, $\tilde q$, $\| b^1/(1+|x|)\|_{L^1_t L^\infty_x}$ and $\| b^2\|_{L^\infty_t \tilde{L}^{d+\eps}_x}$.
\end{lemma}

\begin{proof}
Let $(\tilde p, \tilde q)$ be fixed and denote by $(\tilde p', \tilde q')$ their conjugate exponents. By the duality relation $(L^{\tilde q}_t L^{\tilde p}_x)^\ast = L^{\tilde q'}_t L^{\tilde p'}_x$, in order to prove the claim it suffices to show that
\begin{equation}\label{eq:density-goal}
	|\langle f,\mu\rangle|
	= \bigg| \int_0^T \int_{\R^d} f_s(x) \mu_s(\!\dd x) \dd s\,\bigg|
	= \bigg| \int_0^T \E[f_s(X_s)] \dd s\,\bigg|
	\lesssim \| f\|_{L^{\tilde q'}_t L^{\tilde p'}_x} (1 + \E[|X_0|])
\end{equation}
for all $f\in L^{\tilde q'}_t L^{\tilde p'}_x$; by linearity, we may assume $\| f\|_{L^{\tilde q'}_t L^{\tilde p'}_x}=1$.
Observe that $(\tilde q,\tilde p)$ satisfy \eqref{eq:integrability-density} if and only if their duals satisfy $1/\tilde q' + d/\tilde p'<1$; we can therefore invoke Lemma \ref{lem:decomposition-lemma} to decompose $f=f^{\leq} + f^>$ with $\| f^\leq \|_{L^{1+\eps}_t L^\infty_x}, \| f^>\|_{L^\infty_t L^{d+\eps}_x} \leq 1$. The first term is easy to estimate, since
\begin{equation}\label{eq:density-estim-easy}
	\bigg| \int_0^T \E[f^{\leq}_t(X_t)] \dd t\bigg|
	\leq \int_0^T \| f^{\leq}_t\|_{L^\infty_x} \dd t
	\lesssim_T \| f^\leq\|_{L^{1+\eps}_t L^\infty_x}.
\end{equation}
For the second one, fix any value $\lambda>0$ large enough such that Proposition \ref{prop:zvonkin} applies for $g=b^2$ and $f^>$ in place of $f$; let $u$ denote the associated scalar-valued solution to \eqref{eq:zvonkin-pde}, which thus satisfies 	\eqref{eq:zvonkin-estim}. 
Applying It\^o's formula on $[0,T]$, we find
\begin{align*}
	u_T(X_T)-u_0(X_0)
	= \int_0^T (\partial_t u + \frac{1}{2} a:D^2 u + b^2\cdot\nabla u)(X_t)\dd t + \int_0^T (b^1\cdot\nabla u)(X_t) \dd t + M_T
\end{align*}
for a suitable martingale $M$. Rearranging the terms, applying $u_T\equiv 0$ and taking expectation, we get
\begin{align*}
	 \int_0^T \E[ f^>_t(X_t)] \dd t = \E[u_0(X_0)] + \int_0^T \E[(b^1\cdot\nabla u)(X_t) + \lambda u(X_t)] \dd t;
\end{align*}
applying assumption \eqref{eq:ass-drift} for $b^1$, we then find
\begin{equation}\label{eq:density-estim-hard}\begin{split}
	\bigg| \int_0^T \E [f^>_t(X_t)] \dd t \bigg| 
	& \lesssim_T \| u_0\|_{L^\infty_x} + \| \nabla u\|_{L^\infty_{t,x}} \Big\| \frac{b^1}{1+|x|}\Big\|_{L^1_t L^\infty_x} \Big(1+\E[\|X\|_{C^0_t}]\Big) + \lambda \| u\|_{L^\infty_{t,x}}\\
	& \lesssim 1 + \E[ |X_0|] <\infty
\end{split}\end{equation}
where in the last step we applied Lemma \ref{lem:moments-X}.
Combining \eqref{eq:density-estim-easy} and \eqref{eq:density-estim-hard} yields \eqref{eq:density-goal} and thus the conclusion.
\end{proof}

\section{Proof of the main results}\label{sec:main-thm}

\begin{proof}[Proof of Theorem \ref{thm:main-theorem}]
The proof is based on classical approximation and compactness arguments.
Let $(b,\sigma)$ satisfying Assumption \ref{ass:diffusion} and $\mu_0\in \cP_1$ be given.
By mollifying $b^1$, $b^2$ and $\sigma$, we can construct an approximating sequence $(b^{1,n},b^{2,n},\sigma^n)$ satisfying Assumption \ref{ass:diffusion} uniformly in $n$; more precisely, we require that
\begin{align*}
	\Big\| \frac{b^{1,n}_t}{1+|x|}\Big\|_{L^\infty_x} \leq h_t, \quad \sup_n \| b^{2,n}\|_{L^\infty_t \tilde{L}^{d+\eps}_x}<\infty,
\end{align*}
where the function $h\in L^{1+\eps}_t$ is independent of $n$, while $\sigma^n$ satisfy conditions \eqref{eq:ass-diffusion-1}-\eqref{eq:ass-diffusion-2} for a constant $K$ and a modulus of continuity $\omega_\sigma$ independent of $n$.
Furthermore, the sequence can be constructed so that
\begin{equation}\label{eq:properties-approx}
\lim_{n\to\infty} \sup_{t,x} |\sigma^n(t,x)-\sigma(t,x)|=0, \quad b^{1,n}\to b^1 \text{ in } L^{1+\eps}_t L^p_{loc}, \quad b^{2,n}\to b^2 \text{ in } L^q_t L^{d+\eps}_{loc}
\end{equation}
for all $p,q<\infty$. Finally, for fixed $n$ the coefficients $(b^{1,n},b^{2,n},\sigma^n)$ are regular, in the sense that $b^n\in L^1_t C^1_{loc}$ and satisfying linear growth conditions, while $\sigma^n\in L^\infty_t C^1_x$.

Consider a filtered probability space $(\Omega,\mathcal{F},\{\cF_t\}_{t\geq 0},\PP)$, endowed with some random variables $(\xi,W)$ such that $\cL(\xi)=\mu_0$, $\xi$ is $\cF_0$-measurable and $W$ is a $\cF_t$-Brownian motion.
For any $n$, we can construct classically a strong solution to the SDE
\begin{equation*}
	\dd X^n = b^n_t(X^n_t) \dd t + \sigma^{n}_t(X^n_t) \dd W_t, \quad X^n\vert_{t=0}=\xi.
\end{equation*}
Since $(b^n,\sigma^n)$ satisfy Assumption \ref{ass:diffusion}, all the results from Section \ref{sec:apriori} apply; in particular, setting $\mu^n_t= \cL(X^n_t)$, by Lemmas \ref{lem:moments-X}-\ref{lem:apriori-density} it holds
\begin{equation}\label{eq:uniform-bounds}
	\sup_n \E\big[ \| X^n\|_{C^{\eps/(1+\eps)}_t} \big]<\infty, \quad \sup_n \| \mu^n\|_{L^{\tilde q}_t L^{\tilde p}_x} <\infty\ \quad \forall\, (\tilde q, \tilde p) \text{ satisfying \eqref{eq:integrability-density}}.
\end{equation}
Furthermore, by \eqref{eq:pathwise-estim-X} we have the $\PP$-a.s. bounds
\begin{equation}\label{eq:prelim-bound}
	\| X^n(\omega)\|_{C^0_t} \lesssim 1 + |\xi(\omega)| + \|Z^n(\omega)\|_{C^{\eps/(1+\eps)}_t}
\end{equation}
with constant independent of $n$ and $Z^n=\int_0^\cdot \tilde\sigma^n_t(Y^n_t) \dd W_t$.
By construction, $\tilde{\sigma}^n$ are uniformly bounded, thus the family of r.v.s $\{\|Z^n(\omega)\|_{C^{\eps/(1+\eps)}_t}\}_n$ admits uniformly bounded second moment, making it uniformly integrable.
As the same holds for the single r.v. $|\xi|$, we deduce uniform integrability of the r.v.s appearing on the l.h.s. of \eqref{eq:prelim-bound}, namely
\begin{equation}\label{eq:uniform-integrability}
	\lim_{R\to\infty} \sup_n \E\Big[ \| X^n\|_{C^0_t} \mathbbm{1}_{\| X^n\|_{C^0_t}>R} \Big] = 0.
\end{equation}
The first estimate in \eqref{eq:uniform-bounds}, together with Ascoli-Arzelà's theorem, immediately implies tightness of $\{\cL(X^n)\}_n$ in $C^0_t$, thus also tightness of $\{\cL(\xi,X^n,W)\}_n$ in $\R^d\times C^0_t\times C^0_t$.
By an application of Prokhorov's theorem, we can extract a (not relabelled) subsequence such that $\{\cL(\xi,X^n,W)\}_n$ converge in law;
by Skorokhod's theorem, we can then construct a new probability space $(\tilde \Omega, \tilde \cF,\tilde \PP)$ and a sequence of random variables $(\tilde \xi^n, \tilde X^n, \tilde W^n)$ defined on it such that $\cL(\xi, X^n, W)= \cL(\xi^n,\tilde X^n,\tilde W^n)$ and $(\tilde \xi^n, \tilde X^n,\tilde W^n)\to (\tilde \xi,\tilde X,\tilde W)$ $\tilde\PP$-a.s. in $\R^d\times C^0_t\times C^0_t$.

Standard arguments show that $W$ is a Brownian motion w.r.t. the common filtration $\cG_t=\sigma(\tilde \xi, \tilde X_r,\tilde W_r: r\leq t)$ and that $\cL(\xi)=\mu_0$; additionally observe that, since $\cL(X^n)$ converge weakly to $\cL(X)$ and satisfy the uniform bounds \eqref{eq:uniform-bounds}, by lower semicontinuity of $L^{\tilde q}_t L^{\tilde p}_x$-norms, setting $\mu_t=\cL(X_t)=\cL(\tilde X_t)$, it holds
\begin{equation}\label{eq:integrability_limit}
	\mu\in L^{\tilde q}_t L^{\tilde p}_x, \quad \| \mu\|_{L^{\tilde q}_t L^{\tilde p}_x} \leq \liminf_{n\to\infty} \| \mu^n\|_{L^{\tilde q}_t L^{\tilde p}_x}<\infty.
\end{equation}

It remains to show that $(\tilde\xi, \tilde X, \tilde W)$ is the desired weak solution to the SDE \eqref{eq:intro-sde-detailed}.
In order to do so, it suffices to show that we can pass to the limit in each term in the approximations, namely that
\begin{equation}\label{eq:main-proof-goal}
	\int_0^\cdot b^{i,n}_t(\tilde X^n_t) \dd t \to \int_0^\cdot b^i_t(\tilde X_t) \dd t,
	\quad \int_0^\cdot \sigma^n_t(\tilde X^n_t) \dd \tilde W^n_t \to \int_0^\cdot \sigma_t(\tilde X_t) \dd \tilde W_t
\end{equation}
in probability in $C^0_t$, for $i=1,2$.

We first consider the stochastic integrals in \eqref{eq:main-proof-goal}, which are the easiest.
By construction $\sigma^n\to \sigma$ uniformly in $(t,x)$ and $\tilde X^n\to \tilde X$ $\tilde \PP$-a.s. in $C^0_t$, so that $\sigma^n(\tilde X^n)\to \sigma(\tilde X)$ as well; on the other hand $\tilde W^n\to \tilde W$ in $C^0_t$, and so by applying \cite[Lemma 2.1]{DGHT2011}, we conclude that $\int_0^\cdot \sigma^n(\tilde X^n) \dd \tilde W^n \to \int_0^\cdot \sigma (\tilde X) \dd \tilde W$ in probability.

We claim that, for $i=1,2$, it holds
\begin{equation}\label{eq:convergence-goal}
	\lim_{n\to\infty} \tilde \E\Big[ \int_0^T |b^{i,n}_t(\tilde X^n_t) - b^i_t(\tilde X_t)| \dd t \Big] = 0
\end{equation}
from which \eqref{eq:main-proof-goal} will follow. We only give the proof for \eqref{eq:convergence-goal} for $i=1$, the other case being similar.
In order to prove \eqref{eq:convergence-goal}, we will actually show that, for any given $\delta>0$, it holds
\begin{equation}\label{eq:convergence-goal2}
	\lim_{n\to\infty} \tilde \E\Big[ \int_0^T |b^{1,n}_t(\tilde X^n_t) - b^1_t(\tilde X_t)| \dd t \Big] \leq \delta.
\end{equation}
We divide our analysis in a few substeps.

\emph{Step 1.} We introduce a cutoff function $\psi_R(x):=\psi(|x|/R)$, where $\psi$ is a smooth function satisfying $\psi\equiv 1$ on $[0,1]$ and $\psi\equiv 0$ on $[2,\infty)$ and $R$ is a parameter to be chosen. Correspondingly, we decompose the integral in \eqref{eq:convergence-goal2} as
\begin{equation}\label{eq:convergence-proof1}\begin{split}
	\int_0^T & |b^{1,n}_t(\tilde X^n_t) - b^1_t(\tilde X_t)| \dd t\\
	& \leq \int_0^T \Big( |b^{1,n}_t(1-\psi_R)|(\tilde X^n_t) + |b^1_t(\psi_R-1)|(\tilde X_t) + |(b^{1,n}_t-b^1_t)\psi_R|(\tilde X^n_t)\Big) \dd t\\
	& \quad + \int_0^T [b^1_t \psi_R(\tilde X^n_t) - b^1_t \psi_R(\tilde X_t)]\dd t =: I^{1,n} + I^2 + I^{3,n} + I^{4,n}.
\end{split}\end{equation}
We estimate these terms separately.

\emph{Step 2.} Recall that $\{X^n\}_n$ satisfy the uniform integrability \eqref{eq:uniform-integrability}, so that the same holds for $\tilde{X}^n$ (as well as $\tilde X$). Therefore we can estimate $\tilde \E[I^{1,n}]$ by
\begin{align*}
	\tilde\E\big[I^{1,n} \big]
	& \leq \tilde\E\Big[ \int_0^T |b^{1,n}_t(\tilde X^n_t)| \mathbbm{1}_{|\tilde X^n_t|\geq R} \dd t \Big]\\
	& \leq \tilde\E\Big[ \int_0^T h_t (1+ |\tilde X^n_t|) \mathbbm{1}_{\|\tilde X^n\|_{C^0_t} \geq R} \dd t\Big]
	\leq \| h\|_{L^1_t} \tilde\E\Big[ (1 + \| \tilde X^n\|_{C^0_t}) \mathbbm{1}_{\|\tilde X^n\|_{C^0_t} \geq R} \Big]
\end{align*}
where the last term goes to $0$ as $R\to\infty$, uniformly in $n$, by virtue of \eqref{eq:uniform-integrability}. The same argument works for $I^2$ as well. In particular, we can choose $R$ independent of $n$ such that $\tilde \E[I^{1,n}+I^2]\leq \delta/2.$

\emph{Step 3.} From now on we work with $R$ fixed, determined by Step 2 above.
Let us fix some $p\in [1,\infty)$ large enough such that $1/(1+\eps)+d/p<1$; by contruction of the approximations, it holds $b^{1,n} \psi_R \to b^1 \psi_R$ in $L^{1+\eps}_t L^p_x$; on the other hand, by \eqref{eq:uniform-bounds} the measures $\mu^n$ are uniformly bounded in $L^{(1+\eps)/\eps}_t L^{p'}_x$. It follows that
\begin{align*}
	\lim_{n\to\infty} \tilde\E[I^{3,n}]
	& = \lim_{n\to\infty} \int_0^T \int_{\R^d} |(b^{1,n}_t-b^1_t)\psi_R|(x) \mu^n_t(\! \dd x) \dd t\\
	& \leq \lim_{n\to\infty} \| (b^{1,n}-b^1) \psi_R\|_{L^{1+\eps}_t L^p_x} \| \mu^n\|_{L^{(1+\eps)/\eps}_t L^{p'}_x} = 0.
\end{align*}

\emph{Step 4.} 
It remains to study $I^{4,n}$. Observe that, if $b^1$ were continuous, then $\tilde \E[I^{4,n}]\to 0$ would follow from the property that $\tilde X^n\to \tilde X$ in $C^0_t$ and dominated convergence; if it isn't, we just need to introduce another approximation procedure. To this end, for any another continuous function $\tilde b$, by addition and subtraction we have
\begin{align*}
	I^{4,n} & \leq \int_0^T |(\tilde b_t \psi_R)(\tilde X^n_t) -  (\tilde b_t \psi_R)(\tilde X_t)| \dd t\\
	& \quad + \int_0^T |(b^1_t-\tilde b_t)\psi_R| (\tilde X^n_t) \dd t + \int_0^T |(b^1_t-\tilde b_t)\psi_R| (\tilde X_t) \dd t\\
	& =: J^{1,n} + J^{2,n} + J^{3}.
\end{align*}
For $J^{1,n}$, the previous argument is now rigorous, so that $\tilde\E[J^{1,n}]\to 0$ as $n\to\infty$.
For $J^{2,n}$ and $J^{3}$, fixing $p$ large enough s.t. $1/(1+\eps) + d/p<1$, we may argue as in Step 3 to find
\begin{equation}\label{eq:convergence-proof2}
	\tilde\E[J^{2,n} + J^3]
	\lesssim \| (\tilde b-b^1) \psi_R\|_{L^{1+\eps}_t L^p_x} \Big(\| \mu^n\|_{L^{(1+\eps)/\eps}_t L^{p'}_x} + \| \mu\|_{L^{(1+\eps)/\eps}_t L^{p'}_x}\Big)
	\lesssim \| (\tilde b-b^1) \psi_R\|_{L^{1+\eps}_t L^p_x}
\end{equation}
where in the last passage we used \eqref{eq:uniform-bounds} and \eqref{eq:integrability_limit}.
Since $b^1 \psi_R\in L^{1+\eps}_t L^p_x$ and continuous, compactly supported functions are dense therein, we can choose $\tilde b$ so that the r.h.s. of \eqref{eq:convergence-proof2} is arbitrarily small, in particular smaller than $\delta/2$.

Combining Steps 1-4 above overall yields \eqref{eq:convergence-goal2}, which concludes the proof.
\end{proof}

\begin{proof}[Proof of Corollary \ref{cor:main-cor-1}]
It follows immediately from Lemma \ref{lem:decomposition-lemma}.
\end{proof}

\begin{proof}[Proof of Corollary \ref{cor:main-cor-2}]
Consider the approximations $(b^n, \sigma^n,X^n)$ 
constructed in the proof of Theorem \ref{thm:main-theorem}.
Clearly $\mu^n_t=\cL(X^n_t)$ are now solutions to \eqref{eq:FP} with $(b,a)$ replaced by $(b^n,a^n)$, where $a^n=\sigma^n (\sigma^n)^\ast$, and $\mu^n_t$ converge weakly to $\mu_t=\cL(X_t)$.
The continuity of $t\mapsto \mu_t$ in the weak convergence of measures is a direct consequence of the fact that $X$ has continuous paths.
The fact that $\mu\in L^{\tilde q}_t L^{\tilde p}_x$ was shown in \eqref{eq:integrability_limit}; the claim that $b\mu$, $a\mu\in L^1_t L^1_\loc$ is then a consequence of H\"older's inequality.

It remains to show that \eqref{eq:FP-weak} holds, which can be obtained by passing to the limit in the approximations, namely showing that for any $\varphi\in C^\infty_c$ it holds
\begin{equation}\label{eq:claim_final_cor}\begin{split}
	\lim_{n\to\infty} \int_0^T &\int_{\R^d} \Big( \partial_t \varphi + b^n_t\cdot\nabla \varphi + \frac{1}{2}\sum_{i,j} a^n_{ij} \partial^2_{ij} \varphi\Big)(x) \mu^n_t(\!\dd x) \dd t\\
	&= \int_0^T \int_{\R^d} \Big( \partial_t \varphi + b_t\cdot\nabla \varphi + \frac{1}{2}\sum_{i,j} a_{ij} \partial^2_{ij} \varphi\Big)(x) \mu_t(\!\dd x) \dd t.
\end{split}\end{equation}
Noting that
\begin{align*}
	\int_0^T \int_{\R^d} & \Big( \partial_t \varphi + b^n_t\cdot\nabla \varphi + \frac{1}{2}\sum_{i,j} a^n_{ij} \partial^2_{ij} \varphi\Big)(x) \mu^n_t(\!\dd x) \dd t\\
	& = \tilde\E \Big[\int_0^T \Big( \partial_t \varphi + b^n_t\cdot\nabla \varphi + \frac{1}{2}\sum_{i,j} a^n_{ij} \partial^2_{ij} \varphi\Big)(\tilde X^n_t) \dd t \Big]
\end{align*}
and that a similar relation holds for $\tilde X$, claim \eqref{eq:claim_final_cor} now follows from the same arguments used in the proof of Theorem \ref{thm:main-theorem}.
\end{proof}

\section*{Acknowledgements}
This work originates from some stimulating discussions with Oleg Butkovsky, to whom I'm very thankful, who pointed out the existence of the paper \cite{krylov1} and wondered about the interpretation to give to the condition $1/q+d/p=1$, also in relation to the work \cite{ButGal}.
I also thank the referee for the careful reading of the manuscript and many insightful comments, which improved its quality.

\section*{Funding information}

The author is supported by the SNSF Grant 182565 and by the Swiss State Secretariat for Education, Research and Innovation (SERI) under contract number MB22.00034 through the project TENSE.

\bibliography{myBiblio}{}

\begin{thebibliography}{10}

\bibitem{BFGM2019}
Lisa Beck, Franco Flandoli, Massimiliano Gubinelli, and Mario Maurelli.
\newblock Stochastic {ODE}s and stochastic linear {PDE}s with critical drift:
  regularity, duality and uniqueness.
\newblock {\em Electron. J. Probab.}, 24:Paper No. 136, 72, 2019.

\bibitem{ButGal}
Oleg Butkovsky and Samuel Gallay.
\newblock Weak existence for {SDE}s with singular drifts and fractional or
  {L}\'evy noise beyond the subcritical regime.
\newblock {\em arXiv preprint arXiv:2311.12013}, 2023.

\bibitem{butkovsky2023stochastic}
Oleg Butkovsky, Khoa L{\^e}, and Leonid Mytnik.
\newblock Stochastic equations with singular drift driven by fractional
  {B}rownian motion.
\newblock {\em arXiv preprint arXiv:2302.11937}, 2023.

\bibitem{DGHT2011}
Arnaud Debussche, Nathan Glatt-Holtz, and Roger Temam.
\newblock Local martingale and pathwise solutions for an abstract fluids model.
\newblock {\em Phys. D}, 240(14-15):1123--1144, 2011.

\bibitem{FlIsRu2017}
Franco Flandoli, Elena Issoglio, and Francesco Russo.
\newblock Multidimensional stochastic differential equations with
  distributional drift.
\newblock {\em Transactions of the American Mathematical Society},
  369(3):1665--1688, 2017.

\bibitem{GalGer2022}
Lucio Galeati and M{\'a}t{\'e} Gerencs{\'e}r.
\newblock Solution theory of fractional {SDE}s in complete subcritical regimes.
\newblock {\em arXiv preprint arXiv:2207.03475}, 2022.

\bibitem{GaHaMa2022}
Lucio Galeati, Fabian~A. Harang, and Avi Mayorcas.
\newblock Distribution dependent {SDE}s driven by additive continuous noise.
\newblock {\em Electron. J. Probab.}, 27:Paper No. 37, 38, 2022.

\bibitem{GyoMar2001}
Istv\'{a}n Gy\"{o}ngy and Teresa Mart\'{\i}nez.
\newblock On stochastic differential equations with locally unbounded drift.
\newblock {\em Czechoslovak Math. J.}, 51(126)(4):763--783, 2001.

\bibitem{KarShr}
Ioannis Karatzas and Steven~E. Shreve.
\newblock {\em Brownian motion and stochastic calculus}, volume 113 of {\em
  Graduate Texts in Mathematics}.
\newblock Springer-Verlag, New York, second edition, 1991.

\bibitem{kinzebulatov2023}
Damir Kinzebulatov.
\newblock Form-boundedness and {SDE}s with singular drift.
\newblock {\em arXiv preprint arXiv:2305.00146}, 2023.

\bibitem{krylov1}
N.~V. Krylov.
\newblock On time inhomogeneous stochastic {I}t\^{o} equations with drift in
  {$L_{d+1}$}.
\newblock {\em Ukra\"{\i}n. Mat. Zh.}, 72(9):1232--1253, 2020.

\bibitem{krylov5}
N.~V. Krylov.
\newblock On diffusion processes with drift in {$L^{d+1}$}.
\newblock {\em arXiv:2102.11465}, 2021.

\bibitem{krylov4}
N.~V. Krylov.
\newblock On potentials of {I}t\^o process with drift in {$L^{d+1}$}.
\newblock {\em arXiv:2102.10694}, 2021.

\bibitem{krylov2021strong}
N.~V. Krylov.
\newblock On strong solutions of {I}t\^{o}'s equations with {$\sigma\in W^1_d$}
  and {$b\in L_d$}.
\newblock {\em Ann. Probab.}, 49(6):3142--3167, 2021.

\bibitem{krylov3}
N.~V. Krylov.
\newblock On the heat equation with drift in {$L^{d+1}$}.
\newblock {\em arXiv:2101.00119}, 2021.

\bibitem{krylov2}
N.~V. Krylov.
\newblock Some properties of solutions of {I}t\^{o} equations with drift in
  {$L_{d+1}$}.
\newblock {\em Stochastic Process. Appl.}, 147:363--387, 2022.

\bibitem{krylov2005}
N.~V. Krylov and M.~R\"{o}ckner.
\newblock Strong solutions of stochastic equations with singular time dependent
  drift.
\newblock {\em Probab. Theory Related Fields}, 131(2):154--196, 2005.

\bibitem{ling2021strong}
Chengcheng Ling and Longjie Xie.
\newblock Strong solutions of stochastic differential equations with
  coefficients in mixed-norm spaces.
\newblock {\em Potential Analysis}, pages 1--15, 2021.

\bibitem{RoRoSa2016}
James~C. Robinson, Jos\'{e}~L. Rodrigo, and Witold Sadowski.
\newblock {\em The three-dimensional {N}avier-{S}tokes equations}, volume 157
  of {\em Cambridge Studies in Advanced Mathematics}.
\newblock Cambridge University Press, Cambridge, 2016.
\newblock Classical theory.

\bibitem{RocZha2}
Michael R{\"o}ckner and Guohuan Zhao.
\newblock {SDE}s with critical time dependent drifts: strong solutions.
\newblock {\em arXiv preprint arXiv:2103.05803}, 2021.

\bibitem{XXZZ2020}
Pengcheng Xia, Longjie Xie, Xicheng Zhang, and Guohuan Zhao.
\newblock {$L^q(L^p)$}-theory of stochastic differential equations.
\newblock {\em Stochastic Process. Appl.}, 130(8):5188--5211, 2020.

\bibitem{XieZhang}
Longjie Xie and Xicheng Zhang.
\newblock Ergodicity of stochastic differential equations with jumps and
  singular coefficients.
\newblock {\em Ann. Inst. H. Poincar\'e Probab. Statist.}, 56:175--229, 2020.

\bibitem{ZhangYuan2021}
Shao-Qin Zhang and Chenggui Yuan.
\newblock A {Z}vonkin's transformation for stochastic differential equations
  with singular drift and applications.
\newblock {\em J. Differential Equations}, 297:277--319, 2021.

\bibitem{Zhang2011}
Xicheng Zhang.
\newblock Stochastic homeomorphism flows of {SDE}s with singular drifts and
  {S}obolev diffusion coefficients.
\newblock {\em Electron. J. Probab.}, 16:no. 38, 1096--1116, 2011.

\bibitem{ZhaZha2021}
Xicheng Zhang and Guohuan Zhao.
\newblock Stochastic {L}agrangian path for {L}eray's solutions of 3{D}
  {N}avier-{S}tokes equations.
\newblock {\em Comm. Math. Phys.}, 381(2):491--525, 2021.

\bibitem{zhao2020}
Guohuan Zhao.
\newblock Stochastic {L}agrangian flows for {SDE}s with rough coefficients.
\newblock {\em arXiv preprint arXiv:1911.05562}, 2020.

\end{thebibliography}
\bibliographystyle{plain}

\end{document}